\documentclass[11pt]{article}
\usepackage{amssymb, amsmath, amsthm, graphicx}
\usepackage[left=1in,top=1in,right=1in]{geometry}
\date{}
\allowdisplaybreaks

\usepackage{tikz}
\usetikzlibrary{backgrounds}
\usetikzlibrary{arrows}
\usetikzlibrary{shapes,shapes.geometric,shapes.misc}
\tikzstyle{tikzfig}=[baseline=-0.25em,scale=0.5]
\pgfkeys{/tikz/tikzit fill/.initial=0}
\pgfkeys{/tikz/tikzit draw/.initial=0}
\pgfkeys{/tikz/tikzit shape/.initial=0}
\pgfkeys{/tikz/tikzit category/.initial=0}
\pgfdeclarelayer{edgelayer}
\pgfdeclarelayer{nodelayer}
\pgfsetlayers{background,edgelayer,nodelayer,main}
\tikzstyle{none}=[inner sep=0mm]
\newcommand{\tikzfig}[1]{
{\tikzstyle{every picture}=[tikzfig]
\IfFileExists{#1.tikz}
  {\input{#1.tikz}}
  {
    \IfFileExists{./figures/#1.tikz}
      {\input{./figures/#1.tikz}}
      {\tikz[baseline=-0.5em]{\node[draw=red,font=\color{red},fill=red!10!white] {\textit{#1}};}}
  }}
}

\tikzstyle{every loop}=[]

\tikzstyle{dot}=[fill=white, draw=black, shape=circle,scale=0.5]
\tikzstyle{root}=[fill=black, draw=black, shape=circle,scale=0.5]
\tikzstyle{directed edge}=[draw=black, fill=none, tikzit draw=black, -{stealth}]
\tikzstyle{dash line}=[-, draw=black, fill=none, dashed]

\theoremstyle{definition}
\newtheorem{theorem}{Theorem}[section]
\newtheorem{lemma}[theorem]{Lemma}

\newtheorem{remark}[theorem]{Remark}
\newtheorem{corollary}[theorem]{Corollary}
\newtheorem{example}[theorem]{Example}
\counterwithin{equation}{section}

\newcommand{\Tr}{\text{Tr}}
\newcommand{\vol}{\text{vol}}
\newcommand{\sgn}{\text{sgn}}
\newcommand{\adj}{\text{adj}}
\newcommand{\rank}{\text{rank}}

\title{Forest formulas of discrete Green's functions}
\author{Fan Chung\thanks{Department of Mathematics, University of California at San Diego, La Jolla, CA, 92093 USA. Email: {\tt fan@ucsd.edu}.}\and Ji Zeng\thanks{Department of Mathematics, University of California at San Diego, La Jolla, CA, 92093 USA. Email:
{\tt jzeng@ucsd.edu}. Partly supported by NSF grant DMS-1800746.}}

\begin{document}

\maketitle

\begin{abstract}
The discrete Green's functions  are the pseudoinverse (or the inverse) of the Laplacian  (or its variations) of a graph. In this paper, we will give combinatorial interpretations of Green's functions in terms of 
enumerating trees and forests in a graph that will be used to derive further formulas for several graph invariants. For example, we show that the trace of the Green's function $\mathbf{G}$ associated with the combinatorial Laplacian of a connected simple graph $\Gamma$ on $n$ vertices satisfies
\begin{equation*}
    \Tr(\mathbf{G})=\sum_{\lambda_i \neq 0} \frac 1 {\lambda_i}= \frac{1}{n\tau}|\mathbb{F}^*_2|
\end{equation*}
where $\lambda_i$ denotes the eigenvalues of the combinatorial Laplacian, $\tau$ denotes the number of spanning trees, and $\mathbb{F}^*_2$ denotes the set of rooted spanning $2$-forests in $\Gamma$.

We will prove forest formulas for discrete Green's functions for directed and weighted graphs and apply them to study random walks on graphs and digraphs. We derive a forest expression of the hitting time for digraphs, which gives combinatorial proofs to old and new results about hitting times, traces of discrete Green's functions, and other related quantities.
\end{abstract}

\section{Introduction}\label{intro}
The classical matrix-tree theorem of Kirchhoff \cite{kirchhoff1847ueber} states that the number of spanning trees in a graph is equal to the determinant of a principle minor of the combinatorial Laplacian. In this paper, we consider the Green's functions which are the pseudoinverses of the Laplacians of graphs. It is our goal to follow the spirit of Kirchhoff to find the inherent combinatorial meanings for the Green's functions. We will show that Green's functions can be expressed as specific enumerations of trees and forests in the graph.
 
Since discrete Green's functions for graphs were introduced in 2000 \cite{chung2000discrete}, there have been numerous related articles most of which are in applications while relatively few are graph-theoretical. Previously, Chebotarev and Agaev \cite{chebotarev2002forest} considered two kinds of pseudoinverses of the combinatorial Laplacian $L$ for a directed graph (shortly digraph). The two inverses, called the group inverse $L^\#$ and the Moore-Penrose inverse $L^+$, of a digraph Laplacian are not equal in general (see \cite{chebotarev2002forest} Section 9). In \cite{chebotarev2002forest}, Chebotarev and Agaev derived the group inverse $L^\#$ of the Laplacian $L$ for a weighted digraph as an enumeration of spanning rooted forests. For an undirected graph $\Gamma$, the two types of inverses coincide and are called the Green's function of $\Gamma$ in \cite{chung2000discrete}. Xu and Yau \cite{xu2013discrete} defined Chung-Yau invariants for a simple graph $\Gamma$ and derived an expression of the Green's function in terms of these invariants. This work was further generalized to weighted graphs in \cite{chang2014spanning} and \cite{chang2017chung}. The Chung-Yau invariants can be  computed as an enumeration of a certain class of subgraphs analogous to spanning linear subgraphs of $\Gamma$ (see \cite{xu2013discrete} Section 2). Beveridge \cite{beveridge2016hitting} considered an analogue of the Green's function (which technically isn't either one of the two pseudoinverses considered in \cite{chebotarev2002forest}) for a strongly-connected weighted digraph and proved a formula for it involving the hitting time.

We list three types of Laplacians with their associated Green's functions for graphs and digraphs, weighted or unweighted (while the detailed definitions will be given in Section \ref{pre}): 
\begin{itemize}
\item The combinatorial Laplacian $L=D-A$ where $D$ is the diagonal degree matrix and $A$ is the adjacency matrix. The Moore-Penrose pseudoinverse of $L$ is called the combinatorial Green's function and is denoted by $\mathbf{G}$.
\item The normalized Laplacian $\mathcal{L}= \Pi^{1/2}D^{-1}L\Pi^{-1/2}$, where $\Pi$ is the diagonal matrix of the stationary distribution. For an undirected graph, the formula simplifies as $\mathcal{L}= D^{-1/2}LD^{-1/2}$. The Moore-Penrose pseudoinverse of $\mathcal{L}$ is called the normalized Green's function and is denoted by $\mathcal{G}$.
\item The generalized Laplacian $\mathcal{L}_t = tI+ \mathcal{L}$ with an additional scalar $t>0$. The associated Green's functions $\mathcal{G}_t$ is the inverse of $\mathcal{L}_t$.
 \end{itemize}
 
We remark that the combinatorial Laplacian $L$ arose early in the study of electrical networks \cite{kirchhoff1847ueber} and is instrumental for enumerating various combinatorial objects. The normalized Laplacian $\mathcal{L}$ is closely related to random walks since irreversible Markov chains can be viewed as random walks on a weighted digraph with transition probability matrix $P=D^{-1}A$ and $\mathcal{L}=\Pi^{1/2}(I-P)\Pi^{-1/2}$. The generalized Green's function $\mathcal{G}_t$ has applications related to ranking algorithms (such as Google's Pagerank algorithms). The combinatorial expressions for the generalized Green's functions $\mathcal{G}_t$ involve rooted forests and can be found in \cite{chung2010zhao}, so we will not discuss this type here.
 
Our expressions for Green's functions $\mathbf{G}$ and $\mathcal{G}$ involve rooted or unrooted spanning forests. Readers unfamiliar with these objects are referred to Section \ref{pre}. Figure \ref{fig:intro_example} includes small examples that will be referred to later.
 \begin{figure}[h]
    \centering
    \scalebox{1.5}{\begin{tikzpicture}
	\begin{pgfonlayer}{nodelayer}
		\node [style=dot] (0) at (-7, 5) {};
		\node [style=dot] (1) at (-5, 5) {};
		\node [style=dot] (2) at (-7, 3) {};
		\node [style=dot] (3) at (-5, 3) {};
		\node [style=dot] (4) at (-3, 5) {};
		\node [style=dot] (5) at (-1, 5) {};
		\node [style=dot] (6) at (-3, 3) {};
		\node [style=dot] (7) at (-1, 3) {};
		\node [style=dot] (8) at (1, 5) {};
		\node [style=dot] (9) at (3, 5) {};
		\node [style=dot] (10) at (1, 3) {};
		\node [style=dot] (11) at (3, 3) {};
		\node [style=dot] (12) at (-7, 1) {};
		\node [style=dot] (13) at (-5, 1) {};
		\node [style=dot] (14) at (-7, -1) {};
		\node [style=dot] (15) at (-5, -1) {};
		\node [style=dot] (16) at (-3, 1) {};
		\node [style=dot] (18) at (-3, -1) {};
		\node [style=dot] (19) at (-1, -1) {};
		\node [style=dot] (21) at (3, 1) {};
		\node [style=dot] (22) at (1, -1) {};
		\node [style=none,scale=0.5] (24) at (-6, 2.5) {(i)};
		\node [style=none,scale=0.5] (25) at (-2, 2.5) {(ii)};
		\node [style=none,scale=0.5] (26) at (2, 2.5) {(iii)};
		\node [style=none,scale=0.5] (27) at (-6, -1.5) {(iv)};
		\node [style=none,scale=0.5] (28) at (-2, -1.5) {(v)};
		\node [style=none,scale=0.5] (29) at (2, -1.5) {(vi)};
		\node [style=root] (30) at (-1, 1) {};
		\node [style=root] (31) at (1, 1) {};
		\node [style=root] (32) at (3, -1) {};
	\end{pgfonlayer}
	\begin{pgfonlayer}{edgelayer}
		\draw (0) to (1);
		\draw (0) to (2);
		\draw (2) to (3);
		\draw (3) to (1);
		\draw (2) to (1);
		\draw (4) to (5);
		\draw (6) to (7);
		\draw (6) to (5);
		\draw (8) to (10);
		\draw (9) to (11);
		\draw [style=dash line] (4) to (6);
		\draw [style=dash line] (5) to (7);
		\draw [style=dash line] (8) to (9);
		\draw [style=dash line] (10) to (11);
		\draw [style=dash line] (10) to (9);
		\draw [style=directed edge] (13) to (12);
		\draw [style=directed edge] (12) to (14);
		\draw [style=directed edge] (14) to (15);
		\draw [style=directed edge] (15) to (13);
		\draw [style=directed edge] (14) to (13);
		\draw [style=directed edge] (16) to (18);
		\draw [style=directed edge] (18) to (30);
		\draw [style=directed edge] (19) to (30);
		\draw [style=directed edge] (21) to (31);
		\draw [style=directed edge] (22) to (32);
	\end{pgfonlayer}
\end{tikzpicture}}
    \caption{(i) A graph $\Gamma_1$; (ii) A spanning tree of $\Gamma_1$; (iii) A spanning $2$-forest of $\Gamma_1$. (iv) A digraph $\Gamma_2$; (v) A rooted spanning tree of $\Gamma_2$; (vi) A rooted spanning $2$-forest of $\Gamma_2$. (Roots pictured as solid nodes.)}
    \label{fig:intro_example}
\end{figure}
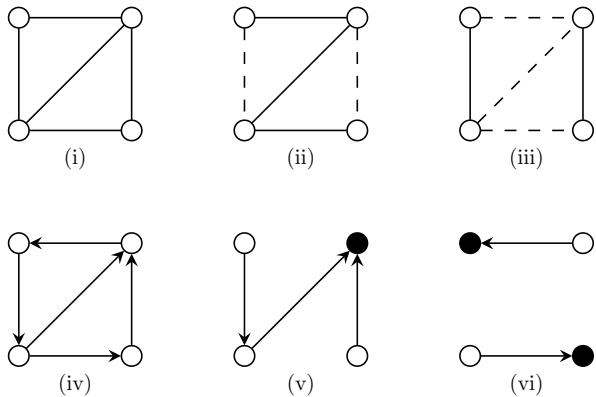
 
In the following forest formulas for $\mathbf{G}$ and $\mathcal{G}$, a spanning $2$-forest is denoted by $T_1\cup T_2$, where $T_1$ and $T_2$ are two subtrees as its two (weakly-)connected components. If $T_1\cup T_2$ is rooted, $r(T_i)$ denotes the root of the subtree $T_i$ respectively. We remark that, throughout this paper, the summations ranges over specified families of spanning $2$-forests, regardless of the order of $T_1$ and $T_2$.
\begin{theorem}\label{main1}
For a weighted strongly-connected digraph $\Gamma$, its combinatorial Green's function $\mathbf{G}$ satisfies
\begin{equation*}
     \mathbf{G}(u,v)=\frac{1}{n\sum_w \tau^2_w}\Big(\hspace{-.12in}\sum_{\substack{T_1\cup T_2\in \mathbb{F}^*_2\\u\in T_1,r(T_1)=v}}\hspace{-.12in}|T_2|\tau_{r(T_2)}\omega(T_1\cup T_2)-\hspace{-.12in}\sum_{\substack{T_1\cup T_2\in \mathbb{F}^*_2\\u\in T_2,r(T_1)=v}}\hspace{-.12in}|T_1|\tau_{r(T_2)}\omega(T_1\cup T_2)\Big),
\end{equation*}
where $\tau_w$ denotes the total weight of rooted spanning trees of $\Gamma$ with root $w$ and $\mathbb{F}^*_2$ denotes the family of rooted spanning $2$-forests in $\Gamma$.
\end{theorem}
\begin{corollary}\label{main1_undirected}
For a  simple connected graph $\Gamma$,  the combinatorial Green's function $\mathbf{G}$ satisfies  \begin{equation*}
    \mathbf{G}(u,v)=\frac{1}{n^2\tau}\Big(\sum_{\substack{T_1 \cup T_2 \in \mathbb{F}_2\\u,v\in T_1}}|T_2|^2-\sum_{\substack{T_1 \cup T_2 \in \mathbb{F}_2\\u\in T_1, v\in T_2}}|T_1||T_2|\Big),
\end{equation*}
where $\tau$ is the number of (unrooted) spanning trees in $\Gamma$ and $\mathbb{F}_2$ denotes the family of (unrooted) spanning $2$-forests in $\Gamma$.
\end{corollary}
\begin{corollary}\label{trace1_undirected}
For  a simple connected  graph $\Gamma$ on $n$ vertices, the trace of the combinatorial Green's function satisfies\begin{equation*}
    \Tr(\mathbf{G})=\frac 1 {n \tau} |\mathbb{F}^*_2|.
\end{equation*}
\end{corollary}

\begin{theorem}\label{main2}
 For a strongly-connected  weighted digraph $\Gamma$,  the normalized Green's function  $\mathcal{G}$ satisfies
 \begin{equation*} 
 \mathcal{G}(u,v)=
 \frac{\sqrt{\pi_u\pi_v}}{\tau_v}
 \Big(\hspace{-.12in}
 \sum_{\substack{T_1\cup T_2\in \mathbb{F}^*_2\\u\in T_1,r(T_1)=v}}
 \hspace{-.12in}\pi(T_2)d_{r(T_2)}
     \omega(T_1\cup T_2)-\hspace{-.2in} \sum_{\substack{T_1\cup T_2\in \mathbb{F}^*_2\\u\in T_2,r(T_1)=v}}
     \hspace{-.12in}\pi(T_1)d_{r(T_2)}\omega(T_1\cup T_2)\Big),
\end{equation*}
where $\pi(T):=\sum_{z\in T}\pi_z$, the sum of stationary distributions of vertices in $T$.
\end{theorem}
\begin{corollary}\label{main2_undirected}
If $\Gamma$ is a connected simple graph,  we have \begin{equation*}
    \mathcal{G}(u,v)=\frac{\sqrt{d_ud_v}}{\vol^2(\Gamma)\tau}\Big(\sum_{\substack{T_1 \cup T_2 \in \mathbb{F}_2\\u,v\in T_1}}\vol^2(T_2)-\sum_{\substack{T_1 \cup T_2 \in \mathbb{F}_2\\u\in T_1, v\in T_2}}\vol(T_1)\vol(T_2)\Big).
\end{equation*}
\end{corollary}

\begin{corollary}
\label{trace2}
For a weighted strongly-connected digraph $\Gamma$, the normalized Green's function $\mathcal{G}$ satisfies
 \begin{equation*} 
  \Tr(\mathcal{G})=
 \frac 1 {\sum_w d_w \tau_w} \sum_{T_1 \cup T_2 \in \mathbb{F}^*_2}d_{r(T_1)}d_{r(T_2)}\omega(T_1 \cup T_2).
\end{equation*}
\end{corollary}
\begin{corollary}\label{trace2_undirected}
For a connected simple graph $\Gamma$, the trace of the normalized Green's function $\mathcal{G}$ satisfies
\begin{equation*}
   \Tr(\mathcal{G})=\frac 1 { \tau \vol({\Gamma})}\sum_{\substack{T_1 \cup T_2 \in \mathbb{F}_2}} \vol{(T_1)}\vol{(T_2)}.
\end{equation*}
 \end{corollary}

We remark that we stated some of the above corollaries for simple graphs since such expressions are most clean and elegant. One can state similar, and not much more complicated, formulas for weighted undirected graphs. Theorem \ref{main2} involves values of stationary distributions hence, strictly speaking, is not purely an enumeration of forests. Still, we can apply the Markov chain tree theorem (\ref{mctt}) to Theorem \ref{main2} and obtain a combinatorial expression of $\mathcal{G}$ involving only forest enumerations and degrees. Also, the normalized Laplacian $\mathcal{L}$ is somewhat special so that its Moore-Penrose inverse equal to the group inverse (which is not true for $L$). These claims shall be further explained later.

This paper is organized as follows: Section \ref{pre} includes detailed definitions and some useful facts.  Section \ref{section_proof} contains the main proofs of our forest formulas as well as some  examples. In Section 4, we apply our forest formulas to study random walks on graphs and digraphs. We consider the forest expressions for several probabilistic quantities, such as the hitting time, the commute time, Kemeny's constant, etc. and we derive some old or new identities/inequalities for these quantities using the forest interpretations.

\section{Preliminaries}\label{pre}
In this paper, we use $\Gamma=(V,E)$ to denote a strongly-connected weighted digraph, where $V$ and $E$ are its vertex set and edge set respectively. For each edge $e=(u,v)$ we denote its weight by $\omega(e)=\omega(u,v)$. It's required that $\omega(u,v)=0$ for $(u,v)\not\in E$.

For simplicity, we do not allow loops in all (di-)graphs in this paper. We also require the weights to be non-negative, although some of our forest formulas still hold if negative weights are allowed.

In this paper, each digraph shall come with a total order on its vertex set for the purpose of  its matrix representation. The adjacency matrix of $\Gamma$ is defined by \begin{equation*}
    A(u,v):=\begin{cases}
	\omega(u,v)\quad \text{, if $(u,v)\in E$};&\\
	0\quad\text{, otherwise.}
	\end{cases}
\end{equation*}
The transpose of a matrix $M$ is denoted as $M^*$.

For each vertex $u$, we define its (out-)degree to be $d_u:=\sum_{v} \omega(u,v)$ and denote the matrix $D:=\text{Diag}[d_u]$. The volume of a subset $U\subset V$ refers to the quantity $\vol(U):=\sum_{u\in U}d_u$.

A typical random walk on $\Gamma$ is defined by the probability transition matrix  $P:=D^{-1}A$. If $\Gamma$ is strongly-connected, there's a unique vector $(\pi_u)_{u\in V}$,  satisfying $\pi_u>0 $ for all $u$,  $\sum_u \pi_u=1$, and $\pi^*P=\pi^*$ (\cite{doyle1984random}Section 1.3). We will call this vector the stationary distribution. (Although aperiodicity is a necessary condition for the convergence of a random walk, we note that a lazy walk with transition probability matrix $\frac{I+P}{2}$ always converges to $\pi$.) We write $\Pi:=\text{Diag}[\pi_u]$.

For a strongly-connected digraph $\Gamma$, the normalized Laplacian matrix is defined by $\mathcal{L}:=\Pi^{\frac{1}{2}}(I-P)\Pi^{-\frac{1}{2}}$ (see \cite{boley2011commute}). For an undirected weighted graph $\Gamma$, the identity $\Pi=\vol(\Gamma)^{-1}D$ is easy to verify and consequently we have $\mathcal{L}=D^{\frac{1}{2}}(I-P)D^{-\frac{1}{2}}$ as in \cite{chungbook}.

The combinatorial Green's function $\mathbf{G}$ (resp. normalized Green's function $\mathcal{G}$) is defined to be the Moore-Penrose pseudoinverse of the combinatorial Laplacian $L$ (resp. normalized Laplacian $\mathcal{L}$). There are many equivalent definitions of Moore-Penrose pseudoinverse, we refer our readers to \cite{horn2012matrix}7.3.P7: Let $M$ be an $m$-by-$n$ matrix with singular value decomposition $M=V\Sigma W^*$, then its Moore-Penrose pseudoinverse $M^+$ is $W\Sigma^+ V^*$ where $\Sigma^+$ is obtained from $\Sigma$ by replacing each nonzero singular value by its reciprocal and then taking the transpose.

Next, we turn to definitions and notations of rooted or unrooted forests. A $k$-forest is an undirected graph that has $k$ connected components and no cycles. A $1$-forest is a tree. A rooted $k$-forest is a digraph, whose underlying graph structure is a $k$-forest, and each weakly-connected component has a vertex, called a root, such that every vertex in that component has a unique directed path to the root. A rooted $1$-forest is called an in-tree or a rooted tree. A sub-digraph is said to be spanning if it contains all vertices of the ambient digraph.
We denote by $\mathbb{F}^*_k$ the set of spanning rooted $k$-forests of the considered graph or digraph. When we are considering an undirected graph, we use $\mathbb{F}_k$ to denote the set of spanning unrooted $k$-forests.

For a rooted or unrooted $k$-forest $F$ of $\Gamma$, we define its weight as \begin{equation*}
    \omega(F)=\prod_{e\in E(F)}\omega(e)
\end{equation*}and for a vertex $u$ in $\Gamma$, we define \begin{equation*}
    \tau_u=\sum_{T\in \mathbb{F}_1^*;r(T)=u} \omega(T).
\end{equation*}
We remark that when $\Gamma$ is undirected, each vertex in a tree can be chosen as the root, so we have\begin{equation}\label{equal_intree_condition}
    \tau_u=\tau, \forall u\in V
\end{equation} where $\tau$ is the total weight of all spanning (unrooted) trees of $\Gamma$.

We shall use the following identity in our proof, known as the Markov chain tree theorem\cite{leighton1983markov}:\begin{equation}
    \label{mctt} \pi_u=\frac{d_u\tau_u}{\sum_w d_w\tau_w}.
\end{equation}

We end this section by stating two results that will be used in the proofs later. The first one is called the all minors matrix tree theorem. The second one is a well-known trick in linear algebra.
\begin{theorem}
[\cite{chaiken1982combinatorial}]\label{ammt}
Let $\Gamma=(V,E)$ be any weighted digraph with a total order on $V$, and $U,W$ two subsets of $V$ s.t. $|U|=|W|$. Write $U=\{u_1<u_2<\dots<u_k\}$ and $W=\{w_1<w_2<\dots<w_k\}$, then we have\begin{equation*}
    \sgn_V(U)\sgn_V(W)\det(L_{\bar{U},\bar{W}})=\sum_{\mu\in S_k} \sgn(\mu)\hspace{-.12in}\sum_{\substack{T_1\cup\dots\cup T_k\in \mathbb{F}^*_k\\ w_i\in T_{\mu(i)},r(T_i)=u_i}}\hspace{-.12in} \omega(T_1\cup\dots\cup T_k),
\end{equation*}
where $L_{\bar{U},\bar{W}}$ denotes the submatrix of $L$ with rows (resp. columns) restricted to $\bar{U}=V\setminus U$ (resp. $\bar{W}=V\setminus W$) and $\sgn_V(U)$ (similarly $\sgn_V(W)$) is defined as $\sgn_V(U)=(-1)^{\sum_{v_i\in U} i}$. Here $v_i$ means the $i$-th smallest vertex in the total order of $V$.
\end{theorem}

\begin{lemma}[\cite{horn2012matrix}0.8.12.3]\label{determinant_sum_lemma}
Let $M,N$ be $n\times n$ matrices, then\begin{equation*}
    \det(M+N)=\sum_{C\subset [n]} \det(M\xleftarrow{C} N)
\end{equation*}where $M\xleftarrow{C} N$ is the matrix obtained by replacing the $C$-indexed columns of $M$ with the corresponding columns of $N$.
\end{lemma}
\section{Proof of main results}\label{section_proof}
We shall use the following lemma to compute the Moore-Penrose pseudoinverse.
\begin{lemma}\label{pseudoinverse_trick}
Let $M$ be an $n$-by-$n$ matrix of rank $n-1$ with left unit kernel $x$ and right unit kernel $y$, then we have $\forall u,v\in [n]$, \begin{equation}
    \frac{(-1)^{u+v}}{\det(M+xy^*)}\det (M_{\bar{v},\bar{u}})=(xy^*)(v,u). \label{pseudoinverse_trick_1}
\end{equation} Moreover if $M^+$ is the Moore-Penrose pseudoinverse of $M$, we have $\forall u,v\in [n]$, \begin{equation}
    \frac{(-1)^{u+v}}{\det(M+xy^*)}\sum_{c\neq i}\det ((M\xleftarrow{c} xy^*)_{\bar{v},\bar{u}})=M^+(u,v).\label{pseudoinverse_trick_2}
\end{equation}
\end{lemma}
\begin{proof}
We deduce (\ref{pseudoinverse_trick_2}) from (\ref{pseudoinverse_trick_1}) first: By definition of Moore-Penrose pseudoinverse, $M^+=(M+xy^*)^{-1}-yx^*$. By Cramer's rule and Lemma~\ref{determinant_sum_lemma},\begin{align*}
    M^+(u,v)&=\frac{(-1)^{u+v}}{\det(M+xy^*)} \det ((M+ xy^*)_{\bar{v},\bar{u}}) -xy^*(v,u)\\
    &=\frac{(-1)^{u+v}}{\det(M+xy^*)}\sum_{C\subset [n]\setminus u}\det ((M\xleftarrow{C} xy^*)_{\bar{v},\bar{u}})-xy^*(v,u).
\end{align*} Notice that $\det ((M\xleftarrow{C} xy^*)_{\bar{v},\bar{u}})=0$ for $|C|>1$ as any two columns of $xy^*$ are dependent, so we only need to consider $|C|=0$ or $1$, \begin{align*}
    M^+(u,v)&=\frac{(-1)^{u+v}}{\det(M+xy^*)}\sum_{c\neq u}\det ((M\xleftarrow{c} xy^*)_{\bar{v},\bar{u}})+\frac{(-1)^{u+v}}{\det(M+xy^*)}\det (M_{\bar{v},\bar{u}})-xy^*(v,u)\\
    &= \frac{(-1)^{u+v}}{\det(M+xy^*)}\sum_{c\neq i}\det ((M\xleftarrow{c} xy^*)_{\bar{v},\bar{u}})\text{, by identity (\ref{pseudoinverse_trick_1})}.
\end{align*}

Now we deduce (\ref{pseudoinverse_trick_1}): The adjugate matrix of $M$ satisfies $\adj(M)M=M\adj(M)=0$. As both the left and right kernel of $M$ has dimension $1$, we have $\adj(M)=\rho yx^*$ for some constant $\rho$. To prove (\ref{pseudoinverse_trick_1}), it suffices to show that $\rho=\det(M+xy^*)$. Indeed,\begin{align*}
    \det(M+xy^*)&= \sum_{u\in [n]} \det(M\xleftarrow{u} xy^*)\text{, by Lemma~\ref{determinant_sum_lemma}}\\
    &=\sum_u \sum_v xy^*(v,u) (-1)^{u+v} \det(M)_{\bar{v},\bar{u}}\\
    &=\sum_{u}\sum_v y_u\adj(M)(u,v)x_v\\
    &=y^* \adj(M) x=y^*(\rho yx^*)x=\rho,
\end{align*}as both $x$ and $y$ are unit vectors.
\end{proof}

We shall prove Theorem \ref{main1} in detail and sketch a proof for Theorem \ref{main2}, as the latter is just a variation of the former proof.
\begin{proof}[Proof of Theorem \ref{main1}]
Our goal is to prove the following identity:\begin{equation}
    \label{eq:main1_0}
    \mathbf{G}(u,v)=\frac{1}{(n\sum_w \tau^2_w)}\sum_{a\neq u}\sum_{b\neq v} \tau_b(\sum_{\substack{T_1\cup T_2\in \mathbb{F}^*_2\\u\in T_1,r(T_1)=v\\a\in T_2,r(T_2)=b }}\hspace{-.12in}\omega(T_1\cup T_2)-\hspace{-.12in}\sum_{\substack{T_1\cup T_2\in \mathbb{F}^*_2\\a\in T_1,r(T_1)=v\\u\in T_2,r(T_2)=b}}\hspace{-.12in}\omega(T_1\cup T_2)).
\end{equation}Since for each rooted $2$-forest with $v$ being one of the roots, the choice of $a$ is the size of the component that does not contain $u$, this identity indeed implies Theorem~\ref{main1}.

We apply Lemma \ref{pseudoinverse_trick} by taking $M=L$. From the definitions, the left unit null vector of $M$ is $x=\rho^{-1}(\frac{\pi_u}{d_u})_{u\in V}$ where $\rho^2=\sum_{w\in V} \pi_w^2/d_w^2$ and the right unit null vector $y=(\frac{1}{\sqrt{n}},\dots,\frac{1}{\sqrt{n}})^*$. We choose $N:=xy^*$ so that we have $N(b,a)=\rho^{-1}\frac{\pi_b}{d_b}\frac{1}{\sqrt{n}}$.

First, we use Lemma~\ref{determinant_sum_lemma} to compute $\det(M+N)$. Since $\rank(N)=1$ and $\rank(M)=n-1$, $\det(M\xleftarrow{C}N)=0$ for $|C|>1$ or $|C|=0$. Hence,\begin{align}
    \notag\det(M+N)&=\sum_{a}  \det(M\xleftarrow{a}N)=\sum_a\sum_b (-1)^{a+b} N(b,a)\det(M_{\bar{b},\bar{a}})\\
    \notag&=\sum_a\sum_b \rho^{-1}\frac{\pi_b}{d_b}\frac{1}{\sqrt{n}}\sgn_V(a)\sgn_V(b)\det(L_{\bar{b},\bar{a}})\\
    \notag&=\sum_b \rho^{-1}\frac{\pi_b}{d_b}\frac{1}{\sqrt{n}}\sum_a\sum_{\substack{T\in \mathbb{F}^*_1\\r(T)=b,a\in T}}\omega(T)\text{, by Theorem \ref{ammt}}\\
    \notag &=\sum_b \rho^{-1}\frac{\pi_b}{d_b}\sqrt{n}\tau_b\\
    \label{eq:main1_1}&= \frac{\rho^{-1}\sqrt{n}}{\sum_w d_w\tau_w} \sum_b \tau_b^2\text{, by identity (\ref{mctt}).}
\end{align}

Then we compute the numerator in identity (\ref{pseudoinverse_trick_2}),\begin{align}
     \notag&(-1)^{u+v}\sum_{a\neq u}\det((M\xleftarrow{a}N)_{\bar{v},\bar{u}})\\
    \notag=&\sum_{a\neq u}\sum_{b\neq v}\sgn_V(u)\sgn_V(v)\sgn_{\bar{u}}(a)\sgn_{\bar{v}}(b)N(b,a)\det(M_{\overline{\{v,b\}},\overline{\{u,a\}}})\\
    \notag=&\sum_{a\neq u}\sum_{b\neq v}\sgn_V(u)\sgn_V(v)\sgn_{\bar{u}}(a)\sgn_{\bar{v}}(b)\rho^{-1}\frac{\pi_b}{d_b}\frac{1}{\sqrt{n}}\det(L_{\overline{\{v,b\}},\overline{\{u,a\}}})\\
    \label{eq:main1_2}=&\frac{\rho^{-1}}{\sqrt{n}\sum_w d_w\tau_w}\sum_{a\neq u}\sum_{b\neq v}\tau_b\sgn_V(u)\sgn_V(v)\sgn_{\bar{u}}(a)\sgn_{\bar{v}}(b)\det(L_{\overline{\{v,b\}},\overline{\{u,a\}}}),
\end{align} by using (\ref{mctt}) again.

For a fixed tuple of $a\neq u,b\neq v$, we have\begin{equation}
    \label{eq:main1_3} \det(L_{\overline{\{v,b\}},\overline{\{u,a\}}})=\sgn_V(\{v,b\})\sgn_V(\{u,a\})\sum_{F}\sgn(F)\omega(F)
\end{equation} where the $F$ runs over all $2$-forests rooted at $v,b$ with $u,a$ in different components. The sign $\sgn(F)$ is as described in Theorem \ref{ammt}.

{\it Claim:} When $F$ matches $u$ to $v$ and $a$ to $b$, \begin{equation}\label{sign1}
    \sgn_V(u)\sgn_V(v)\sgn_{\bar{u}}(a)\sgn_{\bar{v}}(b)\sgn_V(\{u,a\})\sgn_V(\{v,b\})\sgn(F)=1;
\end{equation}When $F$ matches $a$ to $v$ and $u$ to $b$,\begin{equation}\label{sign2}
    \sgn_V(u)\sgn_V(v)\sgn_{\bar{u}}(a)\sgn_{\bar{v}}(b)\sgn_V(\{u,a\})\sgn_V(\{v,b\})\sgn(F)=-1.
\end{equation}
To see this, if we assume $F$ matches $u$ to $v$ and $a$ to $b$, then $\sgn(F)$ only depends on the relative position of $u,v,a,b$ in the total order of $V$. Suppose, say, $u<v<a<b$, then by definition $\sgn(F)=1$, $\sgn_{\bar{u}}(a)=-\sgn_{V}(a)$ and $\sgn_{\bar{v}}(b)=-\sgn_{V}(b)$. Together with $\sgn_V(\{u,a\})=\sgn_V(u)\sgn_V(a)$ and $\sgn_V(\{v,b\})=\sgn_V(v)\sgn_V(b)$, we conclude (\ref{sign1}). The other case can be proved in a similar way, hence proving the claim.

By combining this sign pattern claim with identities (\ref{eq:main1_1}), (\ref{eq:main1_2}), and (\ref{eq:main1_3}),  we obtain (\ref{eq:main1_0}) as desired.
\end{proof}

\begin{proof}[Sketched proof of Theorem \ref{main2}]
Our goal is to establish the following identity:\begin{equation}
    \label{eq:main2_0}
    \mathcal{G}(u,v)= \frac{\sqrt{\pi_u\pi_v}}{\tau_v}\sum_{a\neq u}\sum_{b\neq v} \pi_a d_b(\hspace{-.12in}\sum_{\substack{T_1\cup T_2\in \mathbb{F}^*_2\\u\in T_1,r(T_1)=v\\a\in T_2,r(T_2)=b }}\hspace{-.12in}\omega(T_1\cup T_2)-\hspace{-.12in}\sum_{\substack{T_1\cup T_2\in \mathbb{F}^*_2\\a\in T_1,r(T_1)=v\\u\in T_2,r(T_2)=b}}\hspace{-.12in}\omega(T_1\cup T_2)).
\end{equation} Since for each rooted $2$-forest with $v$ being one of the roots, the choice of $a$ gives the term $\pi(T_1)$ or $\pi(T_2)$, this identity implies Theorem~\ref{main2}.

We follow the proof of Theorem \ref{main1} with two different choices. When we apply Lemma \ref{pseudoinverse_trick}, we choose $M=\mathcal{L}=\Pi^{\frac{1}{2}} D^{-1}L\Pi^{-\frac{1}{2}}$. Both the left and right unit null vectors of $M$ are $x=y=\pi^{\frac{1}{2}}$. By taking $N:=xy^*$ we have $N(b,a)=\sqrt{\pi_b\pi_a}$.

Another difference occurs when we take determinants of submatrices of $M$. Specifically, for a fixed tuple of $a\neq u,b\neq v$, we have\begin{align*}
    &\det(M_{\bar{b},\bar{a}})=(\prod_c \frac{1}{d_c})\frac{d_b}{\sqrt{\pi_b}}\sqrt{\pi_a} \det(L_{\bar{b},\bar{a}}),\\
    &\det(M_{\overline{\{v,b\}},\overline{\{u,a\}}})=(\prod_c \frac{1}{d_c})\frac{d_b}{\sqrt{\pi_b}}\frac{d_v}{\sqrt{\pi_v}}\sqrt{\pi_a}\sqrt{\pi_u}\det(L_{\overline{\{v,b\}},\overline{\{u,a\}}}).
\end{align*}

Using exactly the same arguments as in the proof of Theorem~\ref{main1}, we can conclude\begin{equation*}
    \mathcal{G}(u,v)=\frac{d_v}{(\sum_w d_w\tau_w)}\sqrt{\frac{\pi_u}{\pi_v}}\sum_{a\neq u}\sum_{b\neq v} \pi_a d_b(\hspace{-.12in}\sum_{\substack{T_1\cup T_2\in \mathbb{F}^*_2\\u\in T_1,r(T_1)=v\\a\in T_2,r(T_2)=b }}\hspace{-.12in}\omega(T_1\cup T_2)-\hspace{-.12in}\sum_{\substack{T_1\cup T_2\in \mathbb{F}^*_2\\a\in T_1,r(T_1)=v\\u\in T_2,r(T_2)=b}}\hspace{-.12in}\omega(T_1\cup T_2)).
\end{equation*}Then we apply (\ref{mctt}) to show (\ref{eq:main2_0}) as desired.
\end{proof}
\begin{remark}
One can apply (\ref{mctt}) to replace terms involving  $\pi$'s in (\ref{eq:main2_0}). The resulting forest formula for $\mathcal{G}$ will be slightly more complicated so that the enumeration depends only on the forests and the degrees of vertices.
\end{remark}
\begin{remark}
For $\mathcal{L}$, as the proof shows, the left null vector space coincides with the right null vector space. In this case, $\mathcal{L}\mathcal{G}=\mathcal{G}\mathcal{L}$ as they are both the projection operator onto the orthogonal complement of $\ker(\mathcal{L})$. Thus, $\mathcal{G}$ is also the group inverse of $\mathcal{L}$ (see \cite{chebotarev2002forest}). In general, we do not have $L\mathbf{G}=\mathbf{G}L$.
\end{remark}

When $\Gamma$ is a simple graph, all its subgraphs have weight $1$. Using $\Pi=\vol(\Gamma)^{-1}D$, identity (\ref{equal_intree_condition}), and free choices of roots, one can easily deduce Corollary \ref{main1_undirected} from Theorem \ref{main1}, Corollary \ref{trace2_undirected} from Corollary \ref{trace2}, and Corollary \ref{main2_undirected} from Theorem \ref{main2}. We briefly present the proofs of Corollary~\ref{trace1_undirected} and Corollary~\ref{trace2}.

\begin{proof}[Proof  of Corollary \ref{trace1_undirected}]
From Corollary \ref{main1_undirected}, we have
\begin{equation*}
    \mathbf{G}(u,u)=\frac{1}{n^2\tau}\sum_{\substack{T_1 \cup T_2 \in \mathbb{F}_2\\u\in T_1}}|T_2|^2.
\end{equation*}
Thus,
\begin{align*}
   \Tr(\mathbf{G})&=\sum_u \mathbf{G}(u,u)\\
   &=\frac{1}{n^2\tau}\sum_{\substack{T_1 \cup T_2 \in \mathbb{F}_2}}(\sum_{u \in T_1}|T_2|^2+\sum_{u \in T_2}|T_1|^2)
   \\
   &=\frac{1}{n^2\tau}\sum_{\substack{T_1 \cup T_2 \in \mathbb{F}_2}}(|T_1|~|T_2|^2+|T_1|^2~|T_2|)\\
   &=\frac{1}{n\tau}\sum_{\substack{T_1 \cup T_2 \in \mathbb{F}_2}}|T_1|~|T_2|\\
   &=\frac{1}{n\tau}\sum_{\substack{T_1 \cup T_2 \in \mathbb{F}^*_2}}1
   \\
    &=
   \frac 1 {n \tau} |\mathbb{F}^*_2|,
   \end{align*}as desired.
\end{proof}

\begin{proof}[Proof  of Corollary \ref{trace2}]
By Theorem \ref{main2},
\begin{align*} 
  \Tr (\mathcal{G})&=
  \sum_u \frac{d_u}{\sum_w d_w \tau_w}\sum_{\substack{T_1 \cup T_2 \in \mathbb{F}^*_2\\u=r( T_1)}} \pi(T_2)d_{r(T_2)}  \omega(T_1 \cup T_2)
\\
&=
  \sum_u \frac{1}{\sum_w d_w \tau_w}\sum_{\substack{T_1 \cup T_2 \in \mathbb{F}^*_2\\u=r( T_1)}} \pi(T_2)d_{r(T_1)}d_{r(T_2)}\omega(T_1 \cup T_2)
\\
&=\frac 1 {\sum_w d_w \tau_w}\sum_{\substack{T_1 \cup T_2 \in \mathbb{F}^*_2}}(\pi(T_2)+\pi(T_1))d_{r(T_1)} d_{r(T_2)}\omega(T_1 \cup T_2)
\\
&=\frac 1 {\sum_w d_w \tau_w} \sum_{T_1 \cup T_2 \in \mathbb{F}^*_2}d_{r(T_1)}d_{r(T_2)}\omega(T_1 \cup T_2),
\end{align*}as desired.
\end{proof}
\begin{example}
For the simple graph $\Gamma_1$ in Figure \ref{fig:intro_example}, the nontrivial eigenvalues of the combinatorial Laplacian
$L$ are $2,4,4$. Therefore the trace of the Green's function $\mathbf{G}$ is
\begin{equation*}
    \Tr(\mathbf{G})=\frac 1 2 + \frac 1 4 + \frac 1 4=1.
\end{equation*}
 It is easily checked that there are $8$ spanning trees and the number of $2$-rooted spanning forests is $32$ in $\Gamma$. The forest formula in Theorem \ref{trace1_undirected} states that 
 \begin{equation*}
    \Tr(\mathbf{G})=\frac 1 {n \tau} 32 =\frac {32}{4 \cdot 8} =1.
\end{equation*}
\end{example}
\begin{figure}[h]
    \centering
    \scalebox{1.5}{\begin{tikzpicture}
	\begin{pgfonlayer}{nodelayer}
		\node [style=dot] (0) at (-7, 5) {$a$};
		\node [style=dot] (1) at (-5, 5) {$b$};
		\node [style=dot] (2) at (-7, 3) {$c$};
		\node [style=dot] (3) at (-5, 3) {$d$};
		\node [style=dot] (4) at (-3, 5) {$a$};
		\node [style=dot] (5) at (-1, 5) {$b$};
		\node [style=dot] (6) at (-3, 3) {$c$};
		\node [style=dot] (7) at (-1, 3) {$d$};
		\node [style=dot] (8) at (1, 5) {$a$};
		\node [style=dot] (9) at (3, 5) {$b$};
		\node [style=dot] (10) at (1, 3) {$c$};
		\node [style=dot] (11) at (3, 3) {$d$};
		\node [style=dot] (12) at (-5, 1) {$a$};
		\node [style=dot] (13) at (-3, 1) {$b$};
		\node [style=dot] (14) at (-5, -1) {$c$};
		\node [style=dot] (15) at (-3, -1) {$d$};
		\node [style=dot] (16) at (-1, 1) {$a$};
		\node [style=dot] (17) at (1, 1) {$b$};
		\node [style=dot] (18) at (-1, -1) {$c$};
		\node [style=dot] (19) at (1, -1) {$d$};
	\end{pgfonlayer}
	\begin{pgfonlayer}{edgelayer}
		\draw [style=directed edge] (0) to (2);
		\draw [style=directed edge] (2) to (1);
		\draw [style=directed edge] (4) to (6);
		\draw [style=directed edge] (6) to (7);
		\draw [style=directed edge] (8) to (10);
		\draw [style=directed edge] (11) to (9);
		\draw [style=directed edge] (14) to (13);
		\draw [style=directed edge] (15) to (13);
		\draw [style=directed edge] (18) to (19);
		\draw [style=directed edge] (19) to (17);
	\end{pgfonlayer}
\end{tikzpicture}}
    \caption{All rooted $2$-forests of $\Gamma_2$ with one root being $b$.}
    \label{fig:computation_example}
\end{figure}
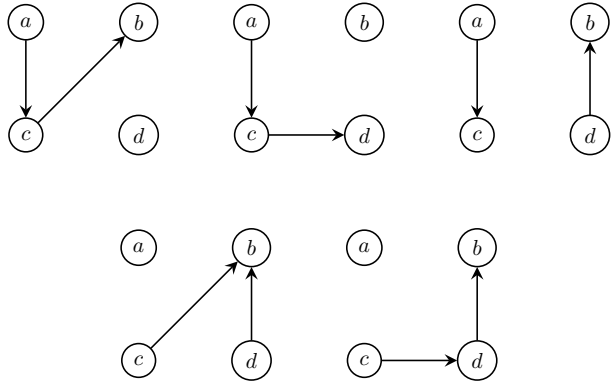
\begin{example}
For the digraph $\Gamma_2$ in Figure~\ref{fig:intro_example}, let's label its vertices as in Figure~\ref{fig:computation_example} and assume each edge has unit weight. It's easy to check that $\tau_a=\tau_b=2$ and $\tau_c=\tau_d=1$. Figure~\ref{fig:computation_example} also lists all the rooted $2$-forests with one root being $b$, then we can read from Theorem \ref{main1} that\begin{equation*}
    \mathbf{G}(a,b)=\frac{1}{4(1^2+1^2+2^2+2^2)}(1-1-2-6-6)=-\frac{7}{20},
\end{equation*}which can also be checked computationally.
\end{example}

\section{Forest formula for hitting times and applications}\label{section_random_walk}
In this section, we will use the forest formula of the Green's function to examine several invariants arising in random walks on graphs or digraphs. First, we will derive the combinatorial expression for the hitting time $H(u,v)$, that is, the expected length of the random walk on $\Gamma$ starting at $u$ and ending at $v$. Using the forest formula of hitting times, we can give combinatorial proofs of old and new results related. For example, we prove the Kemeny's constant of a digraph with $n$ vertices is at least $\frac{n-1}{2}$, which is a tight lower bound achieved by directed cycles.

We state the connection between the Green's function $\mathcal{G}$ and the hitting time $H(u,v)$. The following identity was first established for undirected graphs in \cite{chung2000discrete} and was generalized to strongly-connected digraphs in \cite{boley2011commute}.
\begin{equation}
\label{hit1}
    H(u,v)=\frac{\mathcal{G}(v,v)}{\pi(v)}-\frac{\mathcal{G}(u,v)}{\sqrt{\pi(u)\pi(v)}}.
\end{equation}

The following hitting time formula is also obtained in \cite{chebotarev2020hitting} using group inverse techniques.
\begin{theorem} 
\label{hitting_time}
 In a graph $\Gamma$, we have the following formula for the hitting time $H(u,v)$.
\begin{description}
\item[(i)] If $\Gamma$ is a weighted strongly-connected digraph, then
\begin{equation*}
    H(u,v)=\frac 1 {\tau_v}\sum_{\substack{T_1\cup T_2\in \mathbb{F}^*_2\\u \in T_2, v=r(T_1)}} d_{r(T_2)}\omega(T_1\cup T_2).
\end{equation*} 
\item[(ii)]If additional to (i), $e=(v,u)$ is an edge, then
\begin{equation*}
    H(u,v)= \frac{1}{\tau_v w(e)} \sum_{T \in \mathbb{T}^*_{e}}d_{r(T)} \omega(T),
\end{equation*}
where $\mathbb{T}^*_{e}$ denotes the set of all in-trees containing the edge $e$.
\item[(iii)] If $\Gamma$ is a weighted  connected undirected graph, then
\begin{equation*}
    H(u,v)=\frac 1 {\tau}\sum_{\substack{T_1 \cup T_2 \in \mathbb{F}_2\\u \in T_2, v \in T_1}} \vol(T_2)\omega(T_1 \cup T_2).
\end{equation*}
\end{description}
\end{theorem}
\begin{proof}
From Theorem \ref{main2}, we have 
\begin{equation*}
    \frac{\mathcal{G}(v,v)}{\pi_v}=\frac 1 {\tau_v}\sum_{\substack{T_1\cup T_2\in \mathbb{F}^*_2\\v=r(T_1)}} \pi(T_2)d_{r(T_2)}\omega(T_1\cup T_2).
\end{equation*}
We also have
\begin{equation*}
    \frac{\mathcal{G}(u,v)}{\sqrt{\pi_u \pi_v}}=
    \frac 1 {\tau_v}\Big(\hspace{-.12in}\sum_{\substack{T_1\cup T_2\in \mathbb{F}^*_2\\u\in T_1,v=r(T_1)}}\hspace{-.12in} \pi(T_2)d_{r(T_2)}\omega(T_1\cup T_2)-\hspace{-.12in}\sum_{\substack{T_1\cup T_2\in \mathbb{F}^*_2\\u\in T_2, v=r(T_1)}}\hspace{-.12in} \pi(T_1)d_{r(T_2)}\omega(T_1\cup T_2)\Big).
\end{equation*} 
Since  $\pi(T_1)+\pi(T_2)=1$ when $T_1\cup T_2$ is spanning, we have
\begin{align*}
    H(u,v)&=\frac{\mathcal{G}(v,v)}{\pi_v}- \frac{\mathcal{G}(u,v)}{\sqrt{\pi_u\pi_v}}\\
   &= \frac 1 {\tau_v}\Big(\sum_{\substack{T_1\cup T_2\in \mathbb{F}^*_2\\u\in T_2,v=r(T_1)}} \pi(T_2)d_{r(T_2)}\omega(T_1\cup T_2)+\sum_{\substack{T_1\cup T_2\in \mathbb{F}^*_2\\u\in T_2, v=r(T_1)}} \pi(T_1)d_{r(T_2)}\omega(T_1\cup T_2)\Big)\\
   &=\frac 1 {\tau_v}\sum_{\substack{T_1\cup T_2\in \mathbb{F}^*_2\\u\in T_2,v=r(T_1)}} d_{r(T_2)}\omega(T_1\cup T_2),
    \end{align*} proving (i).

If $e=(v,u)$ is an edge, the above expression can be further simplified since
    $T_1 \cup T_2 \cup e$ forms an in-tree. We have
    \begin{equation*}
    H(u,v)=\frac 1 {\tau_v}\sum_{\substack{T_1\cup T_2\in \mathbb{F}^*_2\\u \in T_2, v=r(T_1)}} d_{r(T_2)}\frac{\omega(T_1\cup T_2\cup e)}{\omega(e)}=\frac 1 {\tau_v\omega(e)}\sum_{\substack{T \in \mathbb{T}^*_{e}}}d_{r(T)}w(T),
    \end{equation*}
proving (ii).

(iii) follows from (i) by  (\ref{equal_intree_condition}) and free choices of roots.
\end{proof}

The commute time $C(u,v)$ is the expected time for a random walk  starting from $u$, reaching $v$, and then returning to $u$. Namely,
\begin{equation*}
    C(u,v)=H(u,v)+H(v,u).
\end{equation*}
The following simple expression of the commute time is an immediate consequence of  Theorem \ref{hitting_time}. A forest expression for commute times of digraphs is still complicated. 
\begin{corollary}
If $\Gamma$ is a connected weighted undirected graph, then
  \begin{equation*}
    C(u,v)=\frac {\vol(\Gamma)} {\tau} \sum_{\substack{T_1 \cup T_2 \in \mathbb{F}_2\\ u \in T_1, v \in T_2}} w(T_1 \cup T_2)
\end{equation*}
If $e=\{u,v\}$ is an edge of $\Gamma$, then
\begin{equation*}
    C(u,v) = \frac {\tau_e}{\tau\omega(e)}\vol(\Gamma)
\end{equation*}
where $\tau_e$ denotes the total weight of all spanning trees containing $e$.
\end{corollary}
The return time is the expected length of a random walk to return to its starting vertex. The following is a classical result about return time. We present a new proof similar to but simpler than \cite{xu2013discrete}Theorem~4.5.
\begin{corollary}\label{return_time}
If $v$ is a vertex of a strongly-connected weighted digraph $\Gamma$, then the expected return time $R(v)$, with starting vertex $v$, is $\frac{\sum_w d_w\tau_w}{d_v\tau_v}=\pi_v^{-1}$.
\end{corollary}
\begin{proof}
Write $N^+_v:=\{u\in V; (v,u)\in E\}$. Using Theorem \ref{hitting_time}, the return time is
\begin{align*}
   R(v)&= 1+\sum_{u\in N^+_v} \frac{\omega(v,u)}{d_v} H(u,v)\\
   &=1+\frac{1}{d_v\tau_v}\sum_{b\neq v} d_b\sum_{\substack{T_1\cup T_2\in \mathbb{F}^*_2\\v=r(T_1),b=r(T_2)}} \sum_{u\in N^+_v\cap T_2} \omega(T_1\cup T_2\cup (v,u)).
\end{align*}It suffices to note that $T_1\cup T_2\cup (v,u)$ forms an in-tree with root $b$ and each such in-tree appears in the summation of RHS exactly once.
Thus, we have
\begin{align*}
    R(v)&=1+\frac{\sum_{b\not = v} d_b \tau_b}{d_v\tau_v}\\
    &= \frac{\sum_b d_b\tau_b}{d_v\tau_v}=\pi_v^{-1},
\end{align*}by identity \eqref{mctt} as wanted.
\end{proof}

Kemeny’s constant is the expected number of  steps required for a random walk starting from a specified vertex to reach a random vertex sampled from the stationary distribution. Surprisingly, this quantity does not depend on which starting state is chosen (\cite{kemeny}, see also \cite{lyons2017probability}Exercise~2.48).
\begin{corollary}
\label{kemeny_trace}
In a graph $\Gamma$, Kemeny's constant $\kappa=\sum_v H(u,v) \pi_v$
  satisfies the following:
\begin{description}
\item[(i)] If $\Gamma$ is a weighted strongly-connected digraph, then
\begin{equation*}
   \kappa=\frac{1}{\sum_wd_w\tau_w}\sum_{\substack{T_1\cup T_2 \in \mathbb{F}^*_2}}d_{r(T_1)}d_{r(T_2)}\omega(T_1\cup T_2)=\Tr(\mathcal{G}).
\end{equation*} 
In particular this quantity is independent of the choice of $u$.
\item[(ii)] If $\Gamma$ is a weighted connected undirected graph, then
\begin{equation*}
    \kappa=\frac{1}{\vol(\Gamma)\tau}\sum_{\substack{T_1\cup T_2 \in \mathbb{F}_2}}\vol(T_1)\vol(T_2)\omega(T_1\cup T_2)=\Tr(\mathcal{G}).
\end{equation*}
\end{description}
\end{corollary}
\begin{proof}
We have, for any vertex $u$,
\begin{align*}
\kappa &= \sum_v H(u,v) \pi_v \\
&= \sum_v 
\frac {\pi_v} {\tau_v} \sum_{\substack{T_1\cup T_2\in \mathbb{F}^*_2\\u \in T_2, v=r(T_1)}} d_{r(T_2)}\omega(T_1\cup T_2) 
\\
&=\sum_v 
\frac 1 {\sum_w d_w \tau_w} \sum_{\substack{T_1\cup T_2\in \mathbb{F}^*_2\\u \in T_2, v=r(T_1)}} d_vd_{r(T_2)}\omega(T_1\cup T_2) 
\\
&=\frac 1 {\sum_w d_w \tau_w} \sum_{\substack{T_1\cup T_2\in \mathbb{F}^*_2}} d_{r(T_1)}d_{r(T_2)}\omega(T_1 \cup T_2)=\Tr(\mathcal{G}),
\end{align*} where the last line is based on Corollary \ref{trace2}.

(ii) follows from (i) by (\ref{equal_intree_condition}) and free choices of roots.
\end{proof}
The forest expression of Kemeny's constant was mentioned in \cite{chebotarev2020hitting} but the equality to the trace of the normalized Green's function was not.

The following is an analogue of Corollary \ref{kemeny_trace} for the combinatorial Green's function.
\begin{corollary}\label{kemeny_trace2}
If $\Gamma$ is a connected simple graph, we have\begin{equation*}
    \sum_u\sum_v H(u,v)=\frac{\vol(\Gamma)}{\tau}\sum_{T_1\cup T_2\in \mathbb{F}_2}|T_1||T_2|=\vol(\Gamma)n\Tr(\mathbf{G}).
    \end{equation*}
    \end{corollary}
\begin{proof}
By Theorem \ref{hitting_time}, we have
\begin{align*}
\sum_u\sum_v H(u,v)&= \sum_u\sum_v\frac 1 {\tau}\sum_{\substack{T_1\cup T_2\in \mathbb{F}_2\\u \in T_2, v\in T_1}}\vol(T_2)\\
&=\frac1{\tau} \sum_{\substack{T_1\cup T_2\in \mathbb{F}_2}} (\sum_{u\in T_2}\sum_{v\in T_1}\vol(T_2)+\sum_{u\in T_1}\sum_{v\in T_2} \vol(T_1))\\
&=\frac{1}{\tau}\sum_{T_1\cup T_2\in \mathbb{F}_2}|T_1||T_2|(\vol(T_1)+\vol(T_2))\\
&=\frac{\vol(\Gamma)}{\tau}\sum_{T_1\cup T_2\in \mathbb{F}_2}|T_1||T_2|\\
&=\frac{\vol(\Gamma)}{\tau}|\mathbb{F}^*_2|
=\vol(\Gamma)n\Tr(\mathbf{G}),
\end{align*}
by Corollary~\ref{trace1_undirected}.
\end{proof}

We can use the forest formula of the hitting time to establish the following tight lower bound for Kememy's constant.
\begin{theorem}
\label{kemeny_lower}
For a strongly-connected weighted digraph $\Gamma$, its Kemeny's constant is at least $\frac{n-1}{2}$ with equality if and only if $\Gamma$ is a directed cycle.
\end{theorem}
\begin{proof}
We rewrite $\kappa$ in Corollary \ref{kemeny_trace} as
\begin{equation*}
\sum_v H(u,v)\pi(v)=\frac{1}{2\sum_w d_w\tau_w}\sum_{w}\sum_{v}\sum_{\substack{T_1\cup T_2\in \mathbb{F}^*_2\\w=r(T_1), v=r(T_2)}}d_wd_v\omega(T_1\cup T_2),
\end{equation*}
where the factor $\frac{1}{2}$ is because of the latter summation running over each rooted $2$-forest twice. 

It suffices to show that for any $w$,
\begin{equation*}
\sum_{v}\sum_{\substack{T_1\cup T_2\in \mathbb{F}^*_2\\w=r(T_1), v=r(T_2)}}d_wd_v\omega(T_1\cup T_2) \geq (n-1)d_w\tau_w.
\end{equation*}
 For a $2$-forest $T_1\cup T_2$, let $E(r(T_2), T_1)$ denote the set of edges 
 $(r(T_2),z)$ with $z \in T_1$. We have
 \begin{align*}
 &\sum_{v}\sum_{\substack{T_1\cup T_2\in \mathbb{F}^*_2\\w=r(T_1), v=r(T_2)}}d_wd_v\omega(T_1\cup T_2) \\
 \geq&\sum_{v}\sum_{\substack{T_1\cup T_2\in \mathbb{F}^*_2\\w=r(T_1), v=r(T_2)}}\sum_{e\in E(r(T_2), T_1)}d_w\omega(T_1\cup T_2 \cup e) \\
 =&(n-1)d_w\tau_w,
\end{align*} 
since  $T_1\cup T_2\cup e$ forms an in-tree  with root $w$ and  each in-tree $T$ with root $w$
appears in the last summation of above inequality exactly $n-1$ times based on the fact that  deleting each   of the $n-1$ edges of $T$ results a rooted $2$-forest with a unique  edge $e$.

From the above proof, the equality holds if and only if for any $2$-forest $T_1\cup T_2$ rooted at any pair $(w,v)$, all edges leaving $v$ are in $E(r(T_2), T_1)$. This is equivalent to each vertex having only one out-edge. Indeed, suppose $(v,w_1),(v,w_2)$ both present in $\Gamma$. Take an in-tree $T$ with root $v$ and let $w$ be the one in $\{w_1,w_2\}$ having a longer distance to $v$ in $T$ and $w'$ the other. Delete the edge $e$ starting at $w$ in $T$, we obtain a rooted $2$-forest $T_1\cup T_2$ rooted at pair $(w,v)$ with $(v,w')\not \in E(r(T_2), T_1)$. Finally, a strongly-connected digraph with each vertex having only one out-edge must be a directed cycle.
\end{proof}

In a graph representing an electrical network with each edge representing a unit resistor (or a resistor of weight $1/w_e$). For an edge $(u,v)$, the {\it effective resistance} $R_{e \hspace{-.02in}f\hspace{-.02in}f}(u,v)$ is the potential difference induced by injecting a unit current into a vertex $u$ and extracting a unit current from another vertex $v$. 
 The effective resistance can be used to give general bounds to the hitting time (see \cite{chandra1996electrical}). To be precise, $R_{e \hspace{-.02in}f\hspace{-.02in}f}(u,v)$ can be computed by using the Green's function $\mathbf{G}$:
\begin{equation*}
    R_{e \hspace{-.02in}f\hspace{-.02in}f}(u,v)=\mathbf{G}(u,u)+ \mathbf{G}(v,v)- \mathbf{G}(u,v)- \mathbf{G}(v,u).
\end{equation*}
\begin{corollary}\label{eff}
In a connected weighted undirected graph $\Gamma$ with an edge $e=\{u,v\}$, we have 
\begin{equation*}
    R_{e \hspace{-.02in}f\hspace{-.02in}f}(u,v)=\frac{\tau_{e}}{\tau\omega(e)}.
    \end{equation*}
\end{corollary}
\begin{proof}
From Theorem \ref{main1}, we have the weighted version of Corollary \ref{main1_undirected}:
\begin{equation*}
   \mathbf{G}(u,v)= \frac{1}{n^2\tau}\Big(
   \sum_{\substack{T_1\cup T_2\in \mathbb{F}_2\\u\in T_1, r(T_1)=v}}|T_2|^2 \omega(T_1 \cup T_2)
   -\sum_{\substack{T_1\cup T_2\in \mathbb{F}_2\\u\in T_2, v =r(T_1)}}|T_1||T_2|\omega(T_1 \cup T_2) \Big).
   \end{equation*}
Thus, we  we have
\begin{align*}
   \mathbf{G}(v,v) -  \mathbf{G}(u,v)&=
   \frac{1}{n^2\tau} \sum_{\substack{T_1\cup T_2\in \mathbb{F}_2\\v \in T_1}}|T_2|^2 \omega(T_1 \cup T_2)\\
   &\quad\quad -\frac{1}{n^2 \tau}\Big(\sum_{\substack{T_1\cup T_2\in \mathbb{F}_2\\u,v\in T_1}}|T_2|^2 \omega(T_1 \cup T_2)- \sum_{\substack{T_1\cup T_2\in \mathbb{F}_2\\u\in T_2, v \in T_1}}|T_1||T_2| \omega(T_1 \cup T_2)\Big)\\
   &=\frac{1}{n^2\tau} \sum_{\substack{T_1\cup T_2\in \mathbb{F}_2\\u \in T_2,v \in T_1}}|T_2|
   (|T_1|+|T_2|) \omega(T_1 \cup T_2)\\
   &=\frac{1}{n\tau}\sum_{\substack{T_1\cup T_2\in \mathbb{F}_2\\u \in T_2,v \in T_1}}|T_2|\omega(T_1 \cup T_2).
\end{align*}
Similarly,
\begin{equation*}
   \mathbf{G}(u,u) -  \mathbf{G}(v,u) =\frac{1}{n\tau}\sum_{\substack{T_1\cup T_2\in \mathbb{F}_2\\u \in T_2,v \in T_1}}|T_1|\omega(T_1 \cup T_2).
\end{equation*}
Together, we have
\begin{align*}
    R_{e \hspace{-.02in}f\hspace{-.02in}f}(u,v)&= \frac{1}{n\tau}\sum_{\substack{T_1\cup T_2\in \mathbb{F}_2\\u \in T_2,v \in T_1}}(|T_1|+|T_2|) \omega(T_1 \cup T_2)\\
    &=\frac{1}{\tau \omega(e)}\sum_{\substack{T\in \mathbb{T}_e}} \omega(T)\\
   &=\frac{\tau_{e}}{\tau \omega(e)},
\end{align*} where $\mathbb{T}_e$ is the set of all (unrooted) trees containing the edge $e$.
\end{proof}

For digraphs, the expression for $R_{e \hspace{-.02in}f\hspace{-.02in}f}$ is not as elegant. Nevertheless,
we use the forest formula to derive several general upper bounds for hitting time.

Here are generalizations of Corollary~3.4 and Corollary~3.5 in \cite{chang2014spanning}, originally obtained for undirected graphs using Chung-Yau invariants.
\begin{theorem}\label{hittingbounds}
Suppose $\Gamma$ is a strongly-connected weighted digraph.
\begin{description}
\item[(i)]
For an edge  $e=(v,u)$,  we have
\begin{equation*}
H(u,v)\leq \frac{d_v}{\omega(v,u)}(\pi_v^{-1}-1).
\end{equation*}
\item[(ii)]
For distinct vertices $u,v$, let  $R$ denote the set of vertices $b$ such that there exists a direct path from $u$ to $b$ without passing $v$. Then
\begin{equation*}
H(u,v)\leq \max_{b \in R}\frac{d_b}{\omega(b,v)}.
\end{equation*}
\end{description}
\end{theorem}
\begin{proof}
By Theorem \ref{hitting_time}, we have
\begin{align*}
    H(u,v)&=\frac{1}{\tau_v}\sum_{b\neq v} d_b
    \sum_{\substack{T_1\cup T_2\in \mathbb{F}^*_2\\ u\in T_2, v=r(T_1)\\b=r(T_2)}} \omega(T_1\cup T_2)\\
    &=\frac{1}{\tau_v\omega(v,u)}\sum_{b\neq v} d_b
    \sum_{\substack{T_1\cup T_2\in \mathbb{F}^*_2\\ u\in T_2, v=r(T_1)\\b=r(T_2)}} \omega(T_1\cup T_2\cup(v,u))\\
    &=\frac{1}{\tau_v\omega(v,u)}
     \sum_{b \not = v}d_b\sum_{\substack{T\in \mathbb{T}^*_e,b=r(T) }} \omega(T)\\
     &\leq \frac{1}{\tau_v\omega(v,u)}\Big(\sum_b d_b\tau_b-d_v\tau_v\Big)\\
     &= \frac{d_v}{\omega(v,u)}\Big(\frac{\sum_b d_b\tau_b}{d_v\tau_v}-1\Big)\\
     &= \frac{d_v}{\omega(v,u)}\Big(\pi_v^{-1}-1\Big),
\end{align*} proving (i).

For (ii) we assume $(b,v)\in E$ for each $b\in R$, otherwise RHS of the inequality is infinity and there's nothing to prove. Note that any $2$-forest $T_1\cup T_2\in \mathbb{F}_2^*$ with $u\in T_2$, $r(T_1)=v$ and $r(T_2)=b$ guarantees $b\in R$, hence Theorem \ref{hitting_time} gives,\begin{align*}
    H(u,v)&=\frac{1}{\tau_v}\sum_{b\neq v} d_b
    \sum_{\substack{T_1\cup T_2\in \mathbb{F}^*_2\\ u\in T_2, v=r(T_1)\\b=r(T_2)}} \omega(T_1\cup T_2)\\
    &=\frac{1}{\tau_v}\sum_{b\in P} \frac{d_b}{\omega(b,v)}
    \sum_{\substack{T_1\cup T_2\in \mathbb{F}^*_2\\ u\in T_2, v=r(T_1)\\b=r(T_2)}} \omega(T_1\cup T_2\cup (b,v))\\
    &\leq \max_{b\in P}\{\frac{d_b}{\omega(b,v)}\} \frac{1}{\tau_v}\sum_{b\in R}\sum_{\substack{T_1\cup T_2\in \mathbb{F}^*_2\\ u\in T_2, v=r(T_1)\\b=r(T_2)}} \omega(T_1\cup T_2\cup (b,v)).
\end{align*}It suffices to note that $T_1\cup T_2\cup (b,v)$ is an in-tree rooted at $v$ and each such an in-tree appears in the summand of RHS at most once.
\end{proof}

For the case of undirected graphs, the upper bound for the hitting time can be further improved.
\begin{theorem}\label{hitting_time_upper_undirected}
 In a connected weighted undirected  graph $\Gamma$, if $\{u,v\}$ is an edge, then
 \begin{equation*}
    H(u,v)= (\vol(\Gamma)-d_v)R_{e \hspace{-.02in}f\hspace{-.02in}f}(u,v).
\end{equation*}
If $u$ and $v$ are not adjacent, let $P$ denote a path connecting $u$ and $v$ with length $l$. Then we have
\begin{equation*}
    H(u,v)\leq l\vol(\Gamma)\max_{e \in E(P)}R_{e \hspace{-.02in}f\hspace{-.02in}f}(e).
\end{equation*}
\end{theorem}
\begin{proof}
We follow the proof of Theorem \ref{hittingbounds}. If $\{u,v\}$ is an edge, we have
\begin{align*}
    H(u,v)
    &=\frac{1}{\tau~\omega(v,u)}
     \sum_{b \not = v}d_b\sum_{\substack{T\in \mathbb{T}^*_e,b=r(T) }} \omega(T)\\
     &= \sum_{b \not = v}d_b \frac{\tau_e}{\tau~ \omega(v,u)}\\
     &= \Big(\sum_b d_b-d_v\Big)R_{e \hspace{-.02in}f\hspace{-.02in}f}(u,v) \\
     &= \Big(\vol(\Gamma) -d_v\Big)R_{e \hspace{-.02in}f\hspace{-.02in}f}(u,v),
\end{align*} 
by using Theorem \ref{eff}.

If $u$ and $v$ are not adjacent and are connected by a path $P$, from the subadditivity of the hitting time, we have
\begin{equation*}
    H(u,v)\leq \sum_{e \in E(P)} H(e)\leq  l~\vol(\Gamma)\max_{e \in E(P)}R_{e \hspace{-.02in}f\hspace{-.02in}f}(e).
\end{equation*} 
This completes the proof of Theorem \ref{hitting_time_upper_undirected}
\end{proof}

As an immediate consequence of Theorem \ref{hitting_time_upper_undirected}, we have the following upper bound for the maximal hitting time of an undirected graph.

\begin{corollary}
For a weighted connected undirected graph $\Gamma$, the maximal hitting time satisfies
\begin{equation*}
     \max_{u,v} H(u,v) \leq \vol(\Gamma) {\rm diam}(\Gamma) R_{e \hspace{-.02in}f\hspace{-.02in}f}
\end{equation*}
where ${\rm diam}(\Gamma)$ denotes the diameter of $\Gamma$ and $R_{e \hspace{-.02in}f\hspace{-.02in}f}$ is the maximum $R_{e \hspace{-.02in}f\hspace{-.02in}f}(e)$ among all edges $e $  in $\Gamma$.
\end{corollary}
For connected simple graphs on $n$ vertices,  the maximum hitting time is defined by
\begin{equation*}
    H^*(n)=\max_{\substack{\Gamma\\|V(\Gamma)|=n}} \max_{u,v} H(u,v),
\end{equation*} 
The upper bound of hitting time given in Theorem \ref{hitting_time_upper_undirected} gives an immediate upper bound of order $\Omega(n^3)$ for the maximum hitting time. In  \cite{brightwell1990maximum}, Winkler and Brightwell showed that $H^*(n)$ is of order $\frac{4}{27}n^3$  and determined the extremal graphs 
 as the lollipop graphs. 

For unweighted digraphs, the story is quite different.
We will use the example given in \cite{aksoychung2016} that achieves the maximum principal ratio
$\gamma(n)$ that is defined to be the ratio of the maximum to minimum values in the stationary distribution in a strongly-connected digraph on $n$ vertices with all edge-weights being $1$. It was shown in \cite{aksoychung2016} that $\gamma(n)= \big(\frac 2 3 + o(1) \big)(n-1)!$.

We consider a digraph $\Gamma(n)$ with vertex set $\{v_1, v_2, \ldots, v_n\}$ and edge set
\begin{equation*}
    E(\Gamma(n))=\{ (v_i, v_{i+1}):  1 \leq i \leq n-1 \} \cup \{(v_j, v_i):  1 \leq i < j \leq n-1 \} \cup \{(v_n,v_1) \}.
\end{equation*}
We require $n \geq 3$ and $\Gamma(5)$ is illustrated in Figure \ref{fig:hitting_example}.
\begin{figure}[h]
    \centering
    \scalebox{1.5}{\begin{tikzpicture}
	\begin{pgfonlayer}{nodelayer}
		\node [style=dot] (2) at (-3, 4) {$v_1$};
		\node [style=dot] (3) at (-1, 4) {$v_2$};
		\node [style=dot] (4) at (1, 4) {$v_3$};
		\node [style=dot] (5) at (3, 4) {$v_4$};
		\node [style=dot] (7) at (5, 4) {$v_5$};
	\end{pgfonlayer}
	\begin{pgfonlayer}{edgelayer}
		\draw [style=directed edge] (2) to (3);
		\draw [style=directed edge] (3) to (4);
		\draw [style=directed edge] (4) to (5);
		\draw [style=directed edge] (5) to (7);
		\draw [style=directed edge, bend left=270, looseness=0.75] (4) to (3);
		\draw [style=directed edge, bend left=270, looseness=0.75] (5) to (4);
		\draw [style=directed edge, bend left=270, looseness=0.75] (5) to (3);
		\draw [style=directed edge, bend left=90, looseness=0.75] (3) to (2);
		\draw [style=directed edge, bend left=90, looseness=0.75] (4) to (2);
		\draw [style=directed edge, bend left=90, looseness=0.75] (5) to (2);
		\draw [style=directed edge, bend left=90, looseness=0.75] (7) to (2);
	\end{pgfonlayer}
\end{tikzpicture}}
    \vspace{-0.5in}
    \caption{$\Gamma(5)$.}
    \label{fig:hitting_example}
\end{figure}

Now we choose $u=v_1$ and $v=v_n$. To compute $H(u,v)$, we note that  $\tau_v=1$ and  all in-tree with root $v$ must be a path. However, there are abundant choices for $T_2$ in the forest formula for $H(u,v)$ in Theorem~\ref{hitting_time}. If $v_i$ is the root of $T_2$, then there is some $j \geq i$ such that $T_1$ is a path from $v_{j+1}$ to $v$ and there are at least
$i^{j-i}$ choices for $T_2$. Hence
\begin{equation*}
    H(u,v)\geq \sum_{i=1}^{n-1} i\cdot\sum_{j=i}^{n-1} i^{j-i}.
\end{equation*}
A rough calculation gives a lower bound for the above:
\begin{equation*}
    H(u,v) \geq \max_{i} (i^{n-i})\approx e^{n \log n (1-\frac{\log \log n}{\log n})}.
\end{equation*}
In the other direction, there are at most $(n-1)\cdot n^{n-2}$ rooted $2$-forests on $n$ vertices. To see this, we use the facts that the complete graph has the most rooted $2$-forests and the combinatorial Laplacian of $K_n$ has eigenvalue  $0$ with multiplicity $1$ and eigenvalue $n$ with multiplicity $n-1$. By applying Theorem 2 in \cite{chung1996combinatorial}, the number of rooted $2$-forest in the complete graph on $n$ vertices is exactly $(n-1)\cdot n^{n-2}$.   Now we can use the forest formula in  Theorem \ref{hitting_time} to establish the following crude upper bound for the hitting time $H(u,v)$ :
\begin{equation*}
    H(u,v)\leq n^n=e^{n \log n}.
\end{equation*}
for any two vertices $u$ and $v$ in any directed weighted digraph on $n$ vertices.

\bibliographystyle{plain}
\bibliography{bib}

\begin{thebibliography}{10}

\bibitem{aksoychung2016}
Sinan Aksoy, Fan Chung, and Xing Peng.
\newblock Extreme values of the stationary distribution of random walks on
  directed graphs.
\newblock {\em Advance in Applied Mathematics}, 81:128--155, 2016.

\bibitem{beveridge2016hitting}
Andrew Beveridge.
\newblock {A hitting time formula for the discrete Green's function}.
\newblock {\em Combinatorics, Probability and Computing}, 25(3):362--379, 2016.

\bibitem{boley2011commute}
Daniel Boley, Gyan Ranjan, and Zhi-Li Zhang.
\newblock {Commute times for a directed graph using an asymmetric Laplacian}.
\newblock {\em Linear Algebra and its Applications}, 435(2):224--242, 2011.

\bibitem{brightwell1990maximum}
Graham Brightwell and Peter Winkler.
\newblock Maximum hitting time for random walks on graphs.
\newblock {\em Random Structures \& Algorithms}, 1(3):263--276, 1990.

\bibitem{chaiken1982combinatorial}
Seth Chaiken.
\newblock A combinatorial proof of the all minors matrix tree theorem.
\newblock {\em SIAM Journal on Algebraic Discrete Methods}, 3(3):319--329,
  1982.

\bibitem{chandra1996electrical}
Ashok~K. Chandra, Prabhakar Raghavan, Walter~L. Ruzzo, Roman Smolensky, and
  Prasoon Tiwari.
\newblock The electrical resistance of a graph captures its commute and cover
  times.
\newblock {\em Computational Complexity}, 6(4):312--340, 1996.

\bibitem{chang2017chung}
Xiao Chang and Hao Xu.
\newblock {Chung-Yau invariants and graphs with symmetric hitting times}.
\newblock {\em Journal of Graph Theory}, 85(3):691--705, 2017.

\bibitem{chang2014spanning}
Xiao Chang, Hao Xu, and Shing-Tung Yau.
\newblock Spanning trees and random walks on weighted graphs.
\newblock {\em Pacific Journal of Mathematics}, 273(1):241--255, 2014.

\bibitem{chebotarev2002forest}
Pavel Chebotarev and Rafig Agaev.
\newblock {Forest matrices around the Laplacian matrix}.
\newblock {\em Linear algebra and its applications}, 356(1-3):253--274, 2002.

\bibitem{chebotarev2020hitting}
Pavel Chebotarev and Elena Deza.
\newblock Hitting time quasi-metric and its forest representation.
\newblock {\em Optimization Letters}, 14(2):291--307, 2020.

\bibitem{chungbook}
Fan Chung.
\newblock {\em Spectral Graph Theory}.
\newblock American Mathematics Society, 1997.

\bibitem{chung1996combinatorial}
Fan Chung and Robert~P. Langlands.
\newblock {A combinatorial Laplacian with vertex weights}.
\newblock {\em Journal of Combinatorial Theory, Series A}, 75(2):316--327,
  1996.

\bibitem{chung2000discrete}
Fan Chung and Shing-Tung Yau.
\newblock {Discrete Green's functions}.
\newblock {\em Journal of Combinatorial Theory, Series A}, 91(1-2):191--214,
  2000.

\bibitem{chung2010zhao}
Fan Chung and Wenbo Zhao.
\newblock Pagerank and random walks on graphs.
\newblock In {\em Fete of Combinatorics and Computer Science}, pages 43--62.
  Springer, Berlin, 2010.

\bibitem{doyle1984random}
Peter~G. Doyle and J.~Laurie Snell.
\newblock {\em Random Walks and Electric Networks}.
\newblock American Mathematical Society, 1984.

\bibitem{horn2012matrix}
Roger~A. Horn and Charles~R. Johnson.
\newblock {\em Matrix Analysis}.
\newblock Cambridge university press, 2012.

\bibitem{kemeny}
John~G. Kemeny and J.~Laurie Snell.
\newblock {\em Finite Markov Chains}.
\newblock Van Nostrand, 1960.

\bibitem{kirchhoff1847ueber}
Gustav Kirchhoff.
\newblock {Ueber die Aufl{\"o}sung der Gleichungen, auf welche man bei der
  Untersuchung der linearen Vertheilung galvanischer Str{\"o}me gef{\"u}hrt
  wird}.
\newblock {\em Annalen der Physik}, 148(12):497--508, 1847.

\bibitem{leighton1983markov}
Frank~T. Leighton and Ronald~L. Rivest.
\newblock {\em The Markov chain tree theorem}.
\newblock Massachusetts Institute of Technology, Laboratory for Computer
  Science, 1983.

\bibitem{lyons2017probability}
Russell Lyons and Yuval Peres.
\newblock {\em {Probability on Trees and Networks}}.
\newblock Cambridge University Press, 2017.

\bibitem{xu2013discrete}
Hao Xu and Shing-Tung Yau.
\newblock {Discrete Green's functions and random walks on graphs}.
\newblock {\em Journal of Combinatorial Theory, Series A}, 120(2):483--499,
  2013.

\end{thebibliography}
\end{document}